\documentclass[12pt]{amsart}
\usepackage{amssymb}
\usepackage{amsmath}
\usepackage{amssymb}
\usepackage{amsthm}
\usepackage{enumerate}
\usepackage[all,arc,curve,color,frame]{xy}
\usepackage[utf8]{inputenc}

\DeclareMathOperator{\End}{End}

\DeclareMathOperator{\tr}{tr}
\DeclareMathOperator{\ad}{ad}

\newcommand{\A}{{\mathcal A}}
\newcommand{\B}{{\mathcal B}}

\newcommand{\E}{{\mathcal E}}
\newcommand{\F}{{\mathcal F}}

\newcommand{\I}{{\mathcal I}}

\newcommand{\K}{{\mathcal K}}

\renewcommand{\P}{{\mathcal P}}

\newcommand{\N}{{\mathbb N}}

\newcommand{\C}{\ensuremath{\mathbb{C}}}
\newcommand{\R}{\ensuremath{\mathbb{R}}}

\newcommand{\p}{\partial}

\newcommand{\unit}{{\bf{1}}}

\newcommand{\U}{{\mathcal U}}
\newcommand{\V}{{\mathcal V}}

\newcommand{\x}{{x_0}}

\newcommand{\Z}{\ensuremath{\mathbb{Z}}}
\newcommand*\xbar[1]{%
   \hbox{%
     \vbox{%
       \hrule height 0.5pt % The actual bar
       \kern0.3ex%         % Distance between bar and symbol
       \hbox{%
         \kern-0.2em%      % Shortening on the left side
         \ensuremath{#1}%
         \kern-0.08em%      % Shortening on the right side
       }%
     }%
   }%
}
\newtheorem{lemma}{Lemma}[section]
\newtheorem{proposition}{Proposition}[section]
\newtheorem{theorem}{Theorem}[section]
\newtheorem{corollary}{Corollary}[section]
\begin{document}
%\tableofcontents

\title[Formal oscillatory integrals]
{Formal oscillatory integrals and deformation quantization}
\author[Alexander Karabegov]{Alexander Karabegov}
\address[Alexander Karabegov]{Department of Mathematics, Abilene
Christian University, ACU Box 28012, Abilene, TX 79699-8012}
\email{axk02d@acu.edu}


\maketitle

\begin{abstract}
Following \cite{L} and \cite{KS}, we formalize the notion of an oscillatory integral interpreted as a functional on the amplitudes supported near a fixed critical point $\x$ of the phase function with zero critical value.
We relate to an oscillatory integral two objects, a formal oscillatory integral kernel and the full formal asymptotic expansion at $\x$. The formal asymptotic expansion is a formal distribution supported at $\x$ which is applied to the amplitude.
In \cite{KS} this distribution itself is called a formal oscillatory integral (FOI). We establish a correspondence between the formal oscillatory integral kernels and the FOIs based upon a number of axiomatic properties of a FOI expressed in terms of its formal integral kernel. Then we consider a family of polydifferential operators related to a star product with separation of variables on a pseudo-K\"ahler manifold. These operators evaluated at a point are FOIs. We completely identify their formal oscillatory kernels. \end{abstract}

\section{Introduction}

Several constructions of star products on symplectic manifolds are obtained from full asymptotic expansions of some oscillatory integrals depending on a small parameter $h$. In the resulting expansion the numerical parameter $h$ is replaced with the formal parameter $\nu$.

The most important examples are the asymptotic expansions of the composition formula of Weyl symbols which leads to the Moyal-Weyl star product and of the Berezin transform which leads to the Wick and anti-Wick star products or, more generally, to the Berezin and Berezin-Toeplitz star products (see, e.g., \cite{E} and \cite{KS}).

It turns out that the formal distribution obtained from the asymptotic expansion of an oscillatory integral and supported at a critical point~$\x$ of the phase function still ``remembers" some formal version of the oscillatory integral from which it was obtained. One can recover the full jet of the formal oscillatory integral kernel at $\x$ from such a formal distribution  and the formal distribution itself can be obtained from the full jet of the formal oscillatory integral kernel at $\x$. 

The stationary phase method for a real phase function was completely formalized in \cite{L}.  Namely, the formal oscillatory integral
\begin{equation}\label{E:Leray}
   J(f) = \nu^{-\frac{n}{2}}\int_{(\x)} e^{\nu^{-1} \psi(x)} f(x) dx,
\end{equation}
where $\psi$ is a real phase function on $\R^n$ with a nondegenerate critical point $\x$ and $f \in C^\infty(\R^n)$ is an amplitude supported near $\x$, is expressed in \cite{L} by an explicit formula as a full formal asymptotic series multiplied by the formal oscillatory exponent
\[
       \exp\left\{\nu^{-1} \psi(\x) \right\}.
\]
This formula contains the square root of the determinant of the Hessian of $\psi$ at $\x$ and depends on the index of inertia of the Hessian. Moreover, the numerical parameter in the non-formal oscillatory integral which produces this asymptotic series in \cite{L} is purely imaginary. If the critical value of the phase function is zero, $\psi(\x)=0$, then the functional $f \mapsto J(f)$ is a formal distribution supported at $\x$.

In this paper we extend the formal asymptotic expansion from \cite{L} to the case when the phase function is complex. Given the formal integral (\ref{E:Leray}) with a general complex phase function, we determine the corresponding asymptotic series only up to a sign. However, we are able to fix the sign in the important case when the Hessian of the phase function is of the Hermitian type.

In deformation quantization on a manifold $M$ we encounter various formal polydifferential operators. Given an operator $C(f_1, \ldots,f_l)$ and a point $\x \in M$, we consider the formal distribution
\begin{equation}\label{E:evalp}
     f_1, \ldots, f_l \mapsto C(f_1, \ldots,f_l)(\x)
\end{equation}
on $M^l$ supported at the point $\x^{(l)} := (\x, \ldots, \x) \in M^l$. Sometimes such distributions are given by formal oscillatory integrals. Our primary goal is to identify their formal oscillatory kernels. These formal kernels can be interpreted as quantum geometric objects which encode important information about the star product. A similar quantum geometrical approach to formal oscillatory functions is developed in~\cite{V}.

The formal distribution obtained from a formal oscillatory integral satisfies several axiomatic properties expressed in terms of the formal oscillatory kernel. These properties correspond to the integration by parts and differentiation with respect to the formal parameter of the formal integral. In particular examples, if we have a conjectural oscillatory kernel of a formal distribution, we can verify it by checking the axioms.

In \cite{KS} a list of axioms is given which allows to associate a class of formal oscillatory kernels with a formal distribution determined up to a formal multiplicative constant. In \cite{KS} and in this paper such distributions themselves are called formal oscillatory integrals (FOI), which is a slight abuse of terminology. In this paper we add more axioms which allows to obtain a bijection between the formal oscillatory kernels and the formal oscillatory integrals. The more axioms we add, the tighter is the connection between the formal kernels and the FOIs. One can check by a tedious calculation that the asymptotic expansion of (\ref{E:Leray}) from \cite{L} satisfies our axioms.

The formal oscillatory kernel considered in \cite{KS} and in this paper has a more general format than that in (\ref{E:Leray}). It contains a formal phase function and a formal volume form on a manifold, and the exposition is coordinate independent. 
In \cite{KS} two important examples of FOIs related to deformation quantization with separation of variables on a pseudo-K\"ahler manifold $M$ were considered. They are obtained from the formal Berezin transform and the so called twisted product. Given a star product with separation of variables on a manifold $M$, we consider a family of polydifferential operators related to it. Evaluating these operators at a point $\x \in M$ as in (\ref{E:evalp}), we obtain for each $l \in \N$ a FOI $K^{(l)}$ on $M^l$ at the point $\x^{(l)} \in M^l$ and determine its formal oscillatory integral kernel. It includes the two examples from \cite{KS} as $K^{(1)}$ and $K^{(2)}$, respectively. In the non-formal integrals in the Berezin's quantization formalism the small parameter~$h$ is positive and the Hessian of the phase function is of the Hermitian type, which affects some of the conventions used in our paper.

\medskip

{\bf Acknowledgments.} I want to express my gratitude to Theodore Voronov for an important discussion of the formal stationary phase method which has helped me to revise and improve the exposition.

\section{Formal oscillatory integrals}

Given a vector space $V$, denote by $V((\nu))$ the space of formal vectors
\[
     v = \nu^k v_k + \nu^{k+1} v_{k+1} + \ldots,
\] 
where $k \in \Z$ and $v_r \in V$ for $r \geq k$. Let $M$ be an oriented manifold~\footnote{We require throughout the paper that the manifold $M$ be oriented  in order to deal with volume forms instead of densities.} of dimension $n$ and let $\x$ be a fixed point in $M$. Consider a complex formal phase function $\varphi = \nu^{-1} \varphi_{-1} + \varphi_0 + \ldots$ on $M$ such that $\x$ is a nondegenerate critical point of the function $\varphi_{-1}$ with zero critical value, $\varphi_{-1}(\x)=0$, and a formal volume form $\rho = \rho_0 + \nu \rho_1 + \ldots$ on $M$ such that $\rho_0$ does not vanish at~$\x$. We want to interpret the following formal expression,
\begin{equation}\label{E:Lerayinv}
   J(f) = \nu^{-\frac{n}{2}}\int_{(\x)} e^{\varphi} f \, \rho,
\end{equation}
where $f \in C^\infty(M)((\nu))$, as a formal distribution supported at $\x$ which satisfies formal algebraic properties of an integral and agrees with the stationary phase method. To this end we give an axiomatic definition of what is called a formal oscillatory integral (FOI) in \cite{KS}. Let $\varphi$ and~$\rho$ be as above.

{\it A FOI on the manifold $M$ at the point $\x$ associated with the pair $(\varphi,\rho)$ is a formal distribution $\Lambda = \Lambda_0 + \nu \Lambda_1 + \ldots$ supported at $\x$ such that $\Lambda_0$ is nonzero and the condition
\begin{equation}\label{E:axiom}
    \Lambda(vf + (v \varphi + \mathrm{div}_\rho v)f)=0
\end{equation}
holds for any vector field $v$ and any function $f$ on~$M$.}

In (\ref{E:axiom})  $\mathrm{div}_\rho v$ denotes the divergence of the vector field $v$ with respect to $\rho$ given by the formula
\[
     \mathrm{div}_\rho v = \frac{\mathbb{L}_v \rho}{\rho},
\]
where $\mathbb{L}_v$ is the Lie derivative with respect to $v$. Property (\ref{E:axiom}) of a FOI corresponds to the formal integration by parts of the integral in~(\ref{E:Lerayinv}). Observe that this definition does not use the assumption that $\varphi_{-1}(\x)=0$.

We prove in the Appendix that for every pair $(\varphi,\rho)$ there exists an associated FOI. It was shown in \cite{KS} that two FOIs associated with the same pair $(\varphi,\rho)$ differ by a formal multiplicative constant $c=c_0 + \nu c_1 + \ldots$ with $c_0 \neq 0$. In particular, for each pair $(\varphi,\rho)$ there exists a unique FOI $\Lambda$ associated with it such that $\Lambda(1)=1$.

It is clear that the definition of a FOI depends only on the full jets of $\varphi$ and $\rho$ at $\x$. In particular, we can restrict ourselves to the germ of the pair $(\varphi,\rho)$ at $\x$. We will call two pairs $(\varphi,\rho)$ and $(\hat\varphi,\hat\rho)$ on a neighborhood $U$ of $\x$ {\it equivalent} if there exists a formal function $u \in C^\infty(U)[[\nu]]$ such that
\[
     \hat\varphi = \varphi + u\mbox{ and } \hat\rho = e^{-u}\rho
\]
on $U$. The expression
\[
      v \varphi + \mathrm{div}_\rho v
\]
does not change if we replace the pair $(\varphi,\rho)$ with an equivalent pair. Therefore, if a FOI $\Lambda$ is associated with a pair $(\varphi,\rho)$, it is also associated with any equivalent pair. We see that it is natural to write the pair $(\varphi,\rho)$ in the form of the ``formal oscillatory integral kernel" 
\[
\nu^{-\frac{n}{2}}e^\varphi \rho
\]
of (\ref{E:Lerayinv}), because it is invariant with respect to the transformation $(\varphi,\rho) \mapsto (\varphi+u, e^{-u}\rho)$. We will sometimes refer to the pair $(\varphi,\rho)$ itself as to a formal oscillatory kernel associated with $\Lambda$.

It follows from (\ref{E:axiom}) that $\Lambda_0 \left((v\varphi_{-1})\cdot f\right) = 0$ for any vector field $v$ and any function $f$ on~$M$. Since $\x$ is a nondegenerate critical point of $\varphi_{-1}$ and $\Lambda_0$ is a nontrivial distribution supported at $\x$, one can verify that
\[
      \Lambda_0 = \alpha \delta_{\x},
\]
where $\alpha$ is a nonzero complex number and $\delta_{\x}$ is the Dirac distribution at $\x$, $\delta_{\x} (f) = f(\x)$.

If there is a pair $(\varphi,\rho)$ and an arbitrary formal volume form $\hat\rho = \hat\rho_0 + \nu \hat\rho_1+\ldots$ on a neighborhood of $\x$ such that $\hat\rho_0$ does not vanish at $\x$, then one can find a formal phase $\hat\varphi$ (possibly on a smaller neighborhood of $\x$) such that the pairs $(\varphi,\rho)$ and $(\hat\varphi,\hat\rho)$ are equivalent. Therefore, if we compare the equivalence classes of two pairs, $(\varphi,\rho)$ and $(\hat\varphi,\hat\rho)$, we can always assume that $\rho=\hat\rho$.

In the next section we will show that if two pairs $(\varphi,\rho)$ and $(\hat\varphi,\rho)$ are associated with the same FOI at a point $\x$, then the full jet of $\varphi-\hat\varphi$ at $\x$ is a formal constant.

\section{A pairing on the space of jets of formal functions}

\begin{proposition}\label{P:nondeg}
Let $\Lambda$ be a FOI on a manifold $M$ at a point $\x$. Then the pairing
\begin{equation}\label{E:pair}
   f, g \mapsto \Lambda(fg)
\end{equation}
on $C^\infty(M)((\nu))$ induces a nondegenerate pairing on the space of full jets of formal functions at $\x$.
\end{proposition}
\begin{proof}
Since $\Lambda$ is a formal distribution supported at $\x$, the pairing (\ref{E:pair}) depends only on the full jets of $f$ and $g$ at $\x$.

Let $U \subset M$ be a contractible coordinate chart containing $\x$ with local coordinates $\{x^i\}, i = 1, \ldots,n$, where $n = \dim M$. There exists a pair $(\psi, dx)$, where $\psi \in C^\infty(U)((\nu))$ and $dx = dx^1 \wedge \ldots \wedge dx^n$, which is associated with $\Lambda$. Condition (\ref{E:axiom}) is equivalent to the following one,
\begin{equation}\label{E:axiom2}
      \Lambda\left (\frac{\p f}{\p x^i} + \frac{\p \psi}{\p x^i}f\right)=0
\end{equation}
for all $i$ and $f \in C^\infty(U)((\nu))$.

Let $\mathcal{Z}_k, k \geq 0,$ be the set of formal functions on $U$ that can be represented as products of partial derivatives of $\psi$ of positive order such that the total order of all partial derivatives in the product is~$k$. Thus, 
 \[
 \mathcal{Z}_0 = \{1\}, \mathcal{Z}_1 = \left\{\frac{\p \psi}{\p x^i}\right\}, \mathcal{Z}_2 = \left\{\frac{\p \psi}{\p x^i}\frac{\p \psi}{\p x^j}, \frac{\p^2 \psi}{\p x^i \p x^j}\right\}, \ldots,
 \]
 where the indices are arbitrary. It will be convenient to denote by $Z_{i_1 \ldots i_k}$ any monomial from $\mathcal{Z}_k$ with the indices of the partial derivatives labeled by $i_1, \ldots, i_k$ in any particular order. For example, one can set
 \[
     Z_{i_1 i_2 i_3} := \frac{\p \psi}{\p x^{i_2}} \frac{\p^2 \psi}{\p x^{i_1} \p x^{i_3}} \in \mathcal{Z}_3.
 \]
Given $Z_{i_1 \ldots i_k} \in \mathcal{Z}_k$, we see that $\p Z_{i_1 \ldots i_k}/\p x^{i_{k+1}}$ is a sum of monomials from $\mathcal{Z}_{k+1}$ and 
\[
\frac{\p \psi}{\p x^{i_{k+1}}} \cdot Z_{i_1 \ldots i_k} \in \mathcal{Z}_{k+1}.
 \]
 Suppose that $f$ is a formal function such that 
 \begin{equation}\label{E:assump}
 \Lambda (fg)=0
 \end{equation}
 for any formal function $g$. Fix an arbitrary positive integer $r$. We will prove by induction on $k, 0 \leq k \leq r,$ that
 \begin{equation}\label{E:step}
     \Lambda \left(\frac{\p^r f}{\p x^{i_{1}} \ldots x^{i_k}}Z_{i_{k+1} \ldots i_r} \right)=0
 \end{equation}
 for any indices $i_1, \ldots, i_k$ and any monomial $Z_{i_{k+1} \ldots i_r} \in \mathcal{Z}_{r-k}$.
 When $k=0$,  by setting $g=Z_{i_1 \ldots i_r}$ in (\ref{E:assump}), we see that $\Lambda(f Z_{i_1 \ldots i_r})=0$. Assume that (\ref{E:step}) holds for $k= p-1$. We will prove that it holds for $k=p$.
 We have from (\ref{E:axiom2}) that
  \begin{equation}\label{E:lproduct}
     \Lambda\left( \left(\frac{\p\psi}{\p x^{i_p}}+ \frac{\p}{\p x^{i_p}}\right) \left(\frac{\p^r f}{\p x^{i_{1}} \ldots x^{i_{p-1}}} Z_{i_{p+1}\ldots i_r} \right ) \right)=0
  \end{equation}
 for any indices $i_1, \ldots,i_r$ and any $Z_{i_{p+1}\ldots i_r}\in \mathcal{Z}_{r-p}$. It follows from the induction assumption that
   \[
     \Lambda\left(\frac{\p^r f}{\p x^{i_{1}} \ldots x^{i_{p-1}}} \frac{\p\psi}{\p x^{i_p}}Z_{i_{p+1}\ldots i_r}  \right)=0
 \]
 and
 \[
    \Lambda\left(\frac{\p^r f}{\p x^{i_{1}} \ldots x^{i_{p-1}}} \frac{\p}{\p x^{i_p}} Z_{i_{p+1}\ldots i_r}  \right)=0.
    \]
  Now (\ref{E:lproduct}) implies that
    \[
       \Lambda\left(\frac{\p^r f}{\p x^{i_{1}} \ldots x^{i_{p}}} Z_{i_{p+1}\ldots i_r}  \right)=0,
    \]
  which proves the claim. Thus, for any indices $i_1, \ldots, i_r$ we have that
  \begin{equation}\label{E:zero}
       \Lambda\left(\frac{\p^r f}{\p x^{i_{1}} \ldots x^{i_{r}}}  \right)=0.
 \end{equation}
Suppose that $f = \nu^s f_s + \ldots$ and $f_s$ has a nontrivial full jet at $\x$. Since $\Lambda_0 = \alpha\delta_{\x}$ with $\alpha \neq 0$,  the leading term of (\ref{E:zero}) is
\[
     \nu^s\alpha \frac{\p^r f_s}{\p x^{i_{1}} \ldots x^{i_{r}}} (\x)=0.
\]
Therefore, the full jet of $f_s$ at $\x$ is zero. This contradiction implies the statement of the Proposition.
\end{proof}
\begin{corollary}\label{C:nondeg}
   If two pairs $(\varphi, \rho)$ and $(\hat\varphi,\rho)$  are associated with the same FOI $\Lambda$ at $\x$ on a manifold $M$, then the full jet of $\varphi-\hat\varphi$ at $\x$ is a formal constant.
\end{corollary}
\begin{proof}
We have from (\ref{E:axiom}) that for any formal function $f$ and any vector field $v$ on $M$,
\[
     \Lambda (v(\varphi-\hat\varphi) \cdot f)=0.
\]
Proposition \ref{P:nondeg} implies that the full jet of $v(\varphi-\hat\varphi)$ at $\x$ is zero for any~$v$. Therefore, the full jets of the phase functions $\varphi$ and $\hat\varphi$ at $\x$ differ by an additive formal constant, whence the Corollary follows.
\end{proof}

Thus, if $\Lambda$ is a FOI at $\x$ associated with a pair $(\varphi,\rho)$ and $\rho$ is fixed, then the full jet of $\varphi$ at $\x$ is determined uniquely up to an additive formal constant. In the next section we add another constraint to the definition of a FOI.  It allows to fix the choice of the full jet of $\varphi$ at $\x$ up to an additive constant from~$\C$.

\section{Differentiation of a FOI with respect to the formal parameter}

Let $U$ be a coordinate chart on a manifold $M$ of dimension $n$ with local coordinates $\{x^i\}$, $\varphi = \nu^{-1}\varphi_{-1} + \varphi_0 + \ldots$ be a formal phase function on $U$ such that $\varphi_{-1}$ has a nondegenerate critical point $\x \in U$, and 
\[
\rho = e^u dx^1 \wedge \ldots \wedge dx^n
\]
be a formal volume form on $U$ for some $u \in C^\infty(U)[[\nu]]$.  Assume that  $\Lambda = \Lambda_0 + \nu\Lambda_1+\ldots$ is a FOI at $\x$ associated with the pair $(\varphi,\rho)$. Then $\Lambda_0=\alpha \delta_{\x}$ for some $\alpha \neq 0$.
We have that
\[
      \mathrm{div}_\rho \frac{\p}{\p x^i} = \frac{\p u}{\p x^i}
\]
and (\ref{E:axiom}) is equivalent to the condition
\begin{equation}\label{E:axiomrho}
    \Lambda\left(\frac{\p f}{\p x^i} + \frac{\p (\varphi + u)}{\p x^i}f\right) = 0
\end{equation}
for all $i = 1, \ldots, n$ and $f \in C^\infty(M)$. Using that $\Lambda_0(f)=\alpha f(\x)$, we extract the component of (\ref{E:axiomrho}) at the zeroth degree of $\nu$, arriving at the equality
\begin{equation}\label{E:reccur}
   \Lambda_1\left(\frac{\p \varphi_{-1}}{\p x^i}f\right) = - \alpha\left(\frac{\p f}{\p x^i} + \frac{\p (\varphi_0 + u_0)}{\p x^i} f\right)\Big |_{x=\x}.
\end{equation}
We will be looking for $\Lambda_1$ in the form
\[
    \Lambda_1(f) = \left(\frac{1}{2} A^{kl} \frac{\p^2}{\p x^k \p x^l} + B^k \frac{\p}{\p x^k} + C\right)f \Big |_{x=\x},
\]
where $A^{kl},B^k,$ and $C$ are constants and the matrix $(A^{kl})$ is symmetric. Denote the Hessian matrix of $\varphi_{-1}$ at $\x$ by $(h_{kl})$,
\[
    h_{kl} := \frac{\p^2 \varphi_{-1}}{\p x^k \p x^l}\Big  |_{x=\x},
\]
and denote by $(h^{kl})$ the inverse matrix. Then, taking into account that $\frac{\p \varphi_{-1}}{\p x^i}(\x)=0$ for all $i$, we get from (\ref{E:reccur}) that
\begin{eqnarray*}
  \left (A^{kl} h_{ki} \frac{\p f}{\p x^l} + B^k h_{ki} f \right) \Big |_{x=\x} =  - \alpha\left(\frac{\p f}{\p x^i} + \frac{\p (\varphi_0 + u_0)}{\p x^i} f\right)\Big |_{x=\x}.
\end{eqnarray*}
Thus,
\[
      A^{kl} = -\alpha h^{kl} ,B^k = - \alpha h^{ki}\frac{\p (\varphi_0 + u_0)}{\p x^i},
\]
and the constant $C$ can be arbitrary. We have proved the following lemma.
\begin{lemma}\label{L:Lambda1}
The component $\Lambda_1$ of the formal oscillatory integral $\Lambda$ is given by the formula
\[
    \Lambda_1(f) = -\alpha \left(\frac{1}{2} h^{kl} \frac{\p^2}{\p x^k \p x^l} + h^{ki}\frac{\p (\varphi_0 + u_0)}{\p x^i} \frac{\p}{\p x^k}+K\right)f \Big |_{x=\x},
\]
where $K$ is some constant.
\end{lemma}

Now we will additionally assume that the phase function~$\varphi_{-1}$ has zero critical value at $\x$, $\varphi_{-1}(\x)=0$. 

\begin{proposition}\label{P:diffnu}
   Let $\Lambda$ be a FOI at $\x$ associated with the pair $(\varphi,\rho)$. Then the functional
   \begin{equation}\label{E:tildeLambda}
       \tilde\Lambda(f) =  - \frac{2\nu}{n}\left(\frac{d}{d\nu}\Lambda(f) - \Lambda \left(\frac{df}{d\nu} + \frac{d(\varphi + u)}{d\nu}f \right)\right)
   \end{equation}
is another FOI associated with the pair $(\varphi,\rho)$. Moreover, $\tilde\Lambda_0 = \Lambda_0$.
\end{proposition}
\begin{proof}
We assume that $\Lambda_0 = \alpha \delta_{\x}$, where $\alpha$ is a nonzero constant. First we check that $\tilde\Lambda$ is $\nu$-linear. We have
  \begin{eqnarray*}
      \tilde\Lambda(\nu f) =   - \frac{2\nu}{n}\left(\frac{d}{d\nu}\Lambda(\nu f) - \Lambda \left(\left(\frac{d}{d\nu} + \frac{d(\varphi + u)}{d\nu}\right)(\nu f) \right)\right)=\hskip 2cm \\
        - \frac{2\nu}{n}\left(\Lambda(f) + \nu \frac{d}{d\nu}\Lambda(f) - \Lambda(f) - \nu \Lambda \left(\frac{d f}{d\nu} + \frac{d(\varphi + u)}{d\nu}f \right)\right) = \nu \tilde\Lambda(f).
  \end{eqnarray*}
  Hence, $\tilde\Lambda$ is a formal distribution supported at $\x$. Then we cal\-culate the leading term of $\tilde\Lambda(f)$  using Lemma \ref{L:Lambda1} and the assumption that $\varphi_{-1}(\x)=0$, assuming that $f$ does not depend on~$\nu$. The $\nu$-fil\-tration degree of $\tilde\Lambda(f)$ is at least $-1$. The $\nu$-filtration degree of each of the terms 
  \[
       - \frac{2\nu}{n}\left(\frac{d}{d\nu}\Lambda(f)\right) \mbox{ and } \frac{2\nu}{n}\Lambda\left(\frac{du}{d\nu}f\right)
  \]
  is at least one. The component of $\tilde\Lambda(f)$ of $\nu$-degree $-1$ is
  \[
        \frac{2\nu}{n}\Lambda_0 \left(-\frac{1}{\nu^2} \varphi_{-1} f \right) = 0,
         \]
because $\Lambda_0(\varphi_{-1}f) = \alpha\varphi_{-1}(\x)f(\x)=0$.   The component of $\tilde\Lambda(f)$ of $\nu$-degree zero is
   \begin{eqnarray*}
       \frac{2\nu^2}{n}\Lambda_1 \left(-\frac{1}{\nu^2} \varphi_{-1} f \right) = - \frac{2}{n}\Lambda_1 \left(\varphi_{-1} f \right) =\hskip 3cm\\
       \frac{2\alpha}{n}\left(\frac{1}{2} h^{kl} \frac{\p^2}{\p x^k \p x^l} + h^{ki}\frac{\p (\varphi_0 + u)}{\p x^i} \frac{\p}{\p x^k}+K\right)\left(\varphi_{-1} f \right)\Big |_{x=\x}=\\
       \frac{\alpha}{n} h^{kl} h_{kl} f(\x) = \alpha f(\x),
   \end{eqnarray*}
  because $\varphi_{-1}$ has zero of order two at $\x$. We see that the leading term of $\tilde\Lambda$ is $\tilde\Lambda_0 = \Lambda_0 = \alpha \delta_{\x}$.
  
  It remains to show that $\tilde\Lambda$ satisfies (\ref{E:axiomrho}). The operators
  \[
     \frac{\p}{\p x^i} + \frac{\p (\varphi + u)}{\p x^i}\mbox{ and } \frac{d}{d\nu} + \frac{d(\varphi + u)}{d\nu}  
  \]
commute. We obtain from (\ref{E:tildeLambda}) and (\ref{E:axiomrho}) that
   \begin{eqnarray*}
   \tilde \Lambda\left(\left(\frac{\p}{\p x^i} + \frac{\p (\varphi + u)}{\p x^i}\right)f\right)=\hskip 2.5cm\\
    - \Lambda\left(\left(\frac{d}{d\nu} + \frac{d(\varphi + u)}{d\nu}\right)\left(\frac{\p}{\p x^i} + \frac{\p (\varphi + u)}{\p x^i}\right)f \right)=\\
   - \Lambda\left(\left(\frac{\p}{\p x^i} + \frac{\p (\varphi + u)}{\p x^i}\right)\left(\frac{d}{d\nu} + \frac{d(\varphi + u)}{d\nu}\right)f \right)=0.
    \end{eqnarray*}
    We have verified all the claims of the Proposition.
\end{proof} 

It is important to notice that equation (\ref{E:tildeLambda}) does not depend on the choice of local coordinates. If we change the coordinates, the function $u$ will be modified by adding the logarithm of the Jacobian of the coordinate change (which does not depend on $\nu$), which will not change the derivative $du/d\nu$ in  (\ref{E:tildeLambda}).

  Let $\Lambda$ and $\tilde\Lambda$ be as in Proposition \ref{P:diffnu}. We see that there is a formal constant $a(\nu) = 1 + \nu a_1 + \ldots$ such that $\tilde\Lambda= a(\nu) \Lambda$. We get from formula (\ref{E:tildeLambda}) that
 \begin{equation}\label{E:stronga}
      \frac{d}{d\nu}\Lambda(f) - \Lambda \left(\frac{df}{d\nu} + \left(\frac{d(\varphi + u)}{d\nu}- \frac{n}{2\nu}a(\nu)\right)f \right)=0.
  \end{equation}
  It follows from (\ref{E:stronga}) that there exists a unique formal constant $b(\nu) = \nu b_1 + \ldots$ such that the following equation holds,
  \begin{equation}\label{E:strong}
      \frac{d}{d\nu}\Lambda(f) - \Lambda \left(\frac{df}{d\nu} + \left(\frac{d(\varphi + b(\nu) + u)}{d\nu}- \frac{n}{2\nu}\right)f \right)=0.
 \end{equation}
The FOI $\Lambda$ is associated with the pair $(\varphi + b(\nu),\rho)$ as well. It follows from (\ref{E:strong}) that if we replace $\varphi$ in (\ref{E:stronga}) with $\varphi+ b(\nu)$, then $a(\nu)$ will be replaced with the unit constant and we will have that $\tilde\Lambda=\Lambda$.

{\it We will say that a pair $(\varphi, \rho)$ associated with a FOI $\Lambda$ is strongly associated with it if $\tilde\Lambda=\Lambda$ in Proposition \ref{P:diffnu} or, equivalently, if equation (\ref{E:stronga}) holds with $a=1$.}

Assume that this is the case. There exists a unique FOI $\hat\Lambda$ associated with the pair $(\varphi,\rho)$ such that $\hat\Lambda(1)=1$. Then $\Lambda = \tilde\Lambda= c(\nu)\hat\Lambda$  for some formal constant $c(\nu) = c_0 + \nu c_1 + \ldots$ with $c_0 \neq 0$. Setting $\Lambda= c(\nu)\hat\Lambda$ and $a=1$ in (\ref{E:stronga}), we get that
\[
     \frac{c'(\nu)}{c(\nu)} \hat\Lambda(f) = - \frac{d}{d \nu}\hat\Lambda(f) + \hat\Lambda \left(\frac{df}{d\nu} + \left(\frac{d(\varphi + u)}{d\nu}- \frac{n}{2\nu}\right)f \right).
\]
 Setting $f=1$, we obtain that
  \[
      \frac{c'(\nu)}{c(\nu)} = \hat\Lambda\left(\frac{d(\varphi + u)}{d\nu}\right) - \frac{n}{2\nu}.
  \]
  Thus, $c(\nu)$ and therefore the FOI $\Lambda$ itself are determined up to a nonzero multiplicative constant from $\C$. On the other hand, if $\Lambda$ satisfies (\ref{E:strong}), then $\beta \Lambda$ satisfies (\ref{E:strong}) for any nonzero $\beta \in \C$. 
  
If $\Lambda$ is also strongly associated with another pair $(\hat\varphi,\rho)$, it follows from equation (\ref{E:stronga}) with $a=1$ that
\[
     \Lambda\left(\frac{d(\varphi-\hat\varphi)}{d\nu}f\right)=0
\]
for all $f$. We see from Proposition \ref{P:nondeg} that the full jet of
\[
     \frac{d(\varphi-\hat\varphi)}{d\nu}
\]
at $\x$ is zero. Using Corollary \ref{C:nondeg}, we obtain that the full jet of $\varphi-\hat\varphi$ at $\x$ is a constant from~$\C$. We have proved the following theorem.
  
  \begin{theorem}\label{T:strong}
  Given a formal phase function $\varphi= \nu^{-1}\varphi_{-1}+\varphi_0+\ldots$ on a manifold $M$ such that $\varphi_{-1}$ has a nondegenerate critical point $\x$ with zero critical value and a formal volume form $\rho=\rho_0+\nu\rho_1+\ldots$ such that $\rho_0$ does not vanish at $\x$, there exists a FOI $\Lambda$ at $\x$ strongly associated with the pair $(\varphi,\rho)$. It is determined up to a nonzero multiplicative constant from $\C$. In particular, there is a unique FOI $\Lambda$ strongly associated with that pair with the leading term $\Lambda_0=\delta_{\x}$. If $\Lambda$ is also strongly associated with another pair $(\hat\varphi,\rho)$,
then the full jet of $\varphi-\hat\varphi$ at $\x$ is a constant from $\C$.
\end{theorem}
  
  \section{The leading term of the formal asymptotic expansion}
  
  Let $\psi$ be a real phase function and $\rho$ be a volume form on an oriented manifold $M$ of dimension $n$. Suppose that $\psi$ has one nondegenerate critical point $\x$ in a neighborhood $U$ of $\x$ with zero critical value, $\psi(\x)=0$, and $\rho$ does not vanish at $\x$. If an amplitude $f \in C^\infty(M)$ is supported on $U$, then the oscillatory integral
\begin{equation}\label{E:stph}
    h^{-\frac{n}{2}} \int_M e^{\frac{i}{h}\psi} f\, d\rho,
\end{equation}
where $h$ is a positive numerical parameter, has an asymptotic expansion as $h \to 0$ according to the method of stationary phase. The leading term of this expansion depends only on the Hessian of $\psi$ at $\x$, the value $f(\x)$, and the volume form $\rho(\x)$. These data are defined on the tangent space $T_\x M$ and the leading term of the expansion is given by a model oscillatory integral on $T_\x M$ with a quadratic phase function, one half of the Hessian quadratic form of $\psi$ at $\x$. Let $\{x^i\}$ be local coordinates on $U$ such that $x^i(\x)=0$ and $\{y^i\}$ be the corresponding coordinates on $T_\x M$. Suppose that $g(y)$ is a compactly supported function on $T_{\x}M$ such that $g(0)=f(\x)$ and $\rho = w(x) dx^1 \wedge \ldots \wedge dx^n$. Then the integral (\ref{E:stph}) and the model integral
\[
    h^{-\frac{n}{2}} \int e^{\frac{i}{2h}\frac{\p^2 \psi}{\p x^j \p x^k}(\x) y^j y^k} g(y)  w(\x)dy^1 \wedge \ldots \wedge dy^n
\] 
on $T_\x M$ both have asymptotic expansions as $h \to 0$ with the same leading term. The integral (\ref{E:stph}) was formalized in \cite{L}.
 
We will interpret the formal integral (\ref{E:Lerayinv}) as a specific FOI at $\x$ strongly associated with the pair $(\varphi,\rho)$. 
Denote by $\Lambda$ the unique FOI at $\x$ with the leading term $\Lambda_0=\delta_{\x}$ strongly associated with that pair. It remains to  choose a nonzero constant $\alpha \in \C$ such that $J = \alpha \Lambda$. This choice should be dictated by the leading term of the stationary phase approximation. 

Since our main applications come from Berezin's quantization formalism, we will consider complex phase  functions and use a model Gaussian integral on the tangent space at the critical point. In general, we will determine the formal integral (\ref{E:Lerayinv}) only up to a sign. We will fix that sign in the case when the Hessian of the phase function is of the Hermitian type.

First we want to give coordinate-free interpretations of the Hessian of a function and of a Gaussian integral. 

Given a complex-valued function $f$ on a manifold $M$ with a critical point $\x$, the Hessian of $f$ at $\x$ is a complex symmetric bilinear form on the tangent space $T_{\x}M$. Denote it by $\mathrm{Hess}_{\x}(f)$. It has the following coordinate free definition. If $v, w$ are vector fields on $M$, then $(v w f)(\x)$ depends only on the values of $v$ and $w$ at $\x$ and $\mathrm{Hess}_{\x}(f)$ is determined by the condition that
 \[
    \mathrm{Hess}_{\x}(f)(v(\x),w(\x)) = (v w \varphi_{-1})(\x) = (w v \varphi_{-1})(\x).
 \]

Let $(A_{ij})$ be a positive definite $n \times n$ matrix with constant coefficients. The following formula for the Gaussian integral on $\R^n$ is well known:  
\begin{equation*}
     \int e^{- \frac{1}{2} A_{ij} x^i x^j} dx^1 \ldots dx^n = \sqrt{\frac{(2\pi)^n}{\det A}}.
\end{equation*}
Below we give a coordinate-free interpretation of this formula.

Let $V$ be a real oriented $n$-dimensional vector space, $Q$ be a positive definite quadratic form, and $\tau$ be a translation invariant ``Lebesgue" volume form on $V$. The Hessian of $Q$ is a positive definite symmetric bilinear form on $V$, $\mathrm{Hess}(Q) : V \otimes V \to \R$, such that $\mathrm{Hess}(Q)(v,v) = 2 Q(v)$. Its top exterior power
\[
    \wedge^n \mathrm{Hess}(Q) : \wedge^n V \otimes \wedge^n V \to \R
\]
is also a positive definite bilinear form. Let $q: \wedge^n V \to \R$ be the associated positive definite quadratic form. There exists  a unique oriented volume form $\sqrt{q}: \wedge^n V \to \R$ whose square is $q$. This is the Riemannian volume form corresponding to the metric given by the Hessian. Then for any positive $h$ the following Gaussian integral is absolutely convergent and its value does not depend on $h$,
\begin{equation}\label{E:Gaussint}
     h^{-\frac{n}{2}}\int_V e^{-\frac{1}{h}Q} \tau = (2\pi)^{\frac{n}{2}} \frac{\tau}{\sqrt{q}}.
\end{equation}
If we replace $h$ in (\ref{E:Gaussint}) with the formal parameter $\nu$, we will get a formal integral to which we can assign the same value. If $Q$ is a complex-valued nondegenerate quadratic form, there are two complex volume forms $\pm \sqrt{q}$ and in general there is no way to single out a ``preferred" branch of the square root. Thus, we will assign a numerical value to the following formal Gaussian integral only up to a sign:
\begin{equation}\label{E:formGauss}
    \nu^{-\frac{n}{2}}\int_{\{0\}} e^{-\frac{1}{\nu}Q} \tau = \pm(2\pi)^{\frac{n}{2}} \frac{\tau}{\sqrt{q}}.
\end{equation}
We will use this assignment to interpret the formal integral (\ref{E:Lerayinv}). 

The leading term of the integral (\ref{E:Lerayinv}) can be obtained from a model Gaussian integral on the tangent space $T_\x M$. It depends only on the value $\gamma:= \varphi_0(\x)$, the volume form $\tau:=\rho_0(\x)$, and the Hessian of $\varphi_{-1}$ at $\x$.  The nondegenerate bilinear form 
 \[
     B(v,w) := -\frac{1}{2}\mathrm{Hess}_{\x}(\varphi_{-1})(v,w)
 \]
induces the quadratic form $Q(v):= B(v,v)$. Let $q: \wedge^n T_{\x}M \to \R$ be the quadratic form associated with the bilinear form $\wedge^n B$. The constant $\alpha$ such that $J = \alpha \Lambda$ should be given by the following properly interpreted model formal integral over $T_{\x}M$,
 \begin{equation}\label{E:modint}
     \nu^{-\frac{n}{2}} \int_{\{0\}} e^{- \nu^{-1}Q + \gamma} \tau.
 \end{equation} 
 To this end we use (\ref{E:formGauss}) and interpret (\ref{E:Lerayinv}) as follows,
 \[
       \nu^{-\frac{n}{2}}\int_{\{0\}} e^{\varphi} f \, \rho = \pm(2\pi)^{\frac{n}{2}} e^\gamma\frac{\tau}{\sqrt{q}}\Lambda(f).
 \]

 Now let $g_{kl}$ be a positive definite $m \times m$ Hermitian matrix with constant coefficients and $h$ be a positive numerical parameter. The following formula for the Gaussian integral of the Hermitian type on $\C^m$  is well known:
  \begin{equation}\label{E:Hermite}
       \int e^{-\frac{1}{h}g_{kl} z^k \bar z^l} \frac{1}{m!} \left(\frac{i g_{kl} dz^k \wedge d\bar z^l}{2\pi h}\right)^m = 1.
  \end{equation}
 If we replace $h$ in (\ref{E:Hermite}) with the formal parameter $\nu$,   we will arrive at the following formal assignment,
 \begin{equation}\label{E:formHermite}
      \nu^{-m}  \int_{\{0\}} e^{-\frac{1}{\nu}g_{kl} z^k \bar z^l} \frac{1}{m!} \left(\frac{i g_{kl} dz^k \wedge d\bar z^l}{2\pi}\right)^m = 1,
  \end{equation}
where we can drop the assumption that the matrix $(g_{kl})$ is positive definite. We will assume only that $(g_{kl})$ is nondegenerate.

In order to describe (\ref{E:Hermite}) and (\ref{E:formHermite}) in a coordinate-free fashion, we consider a constant pseudo-K\"ahler form $\Omega := i g_{kl} dz^k \wedge d\bar z^l$. It has a unique potential $K(z,\bar z) := g_{kl} z^k \bar z^l$ for which $\x=0$ is a nondegenerate critical point.
Then (\ref{E:formHermite}) will be written as
\begin{equation}\label{E:Hermnorm}
     \nu^{-m}  \int_{\{0\}} e^{-\frac{1}{\nu}K} \frac{1}{m!} \left(\frac{\Omega}{2\pi}\right)^m = 1.
\end{equation}
Let $M$ be a complex manifold of complex dimension $m$ and $\varphi_{-1}$ be a complex-valued phase function with a  critical point~$\x$ on $M$. 

{\it We say that the Hessian bilinear form $\mathrm{Hess}_{\x}(\varphi_{-1})$ on $T_{\x}M$ is of the Hermitian type if it is of type (1,1) with respect to the complex structure.} In coordinates it means that
\[
     \frac{\p^2 \varphi_{-1}}{\p z^k \p z^p}(\x)=0 \mbox{ and }  \frac{\p^2 \varphi_{-1}}{\p \bar z^l \p \bar z^q}(\x)=0
\]
for all $k,l,p,q$.  We call the $m \times m$ matrix with the entries
\[
       \frac{\p^2 \varphi_{-1}}{\p z^k \p \bar z^l}(\x)
\]
the Hermitian Hessian of $\varphi_{-1}$ at $\x$. Let the critical point $\x$ of $\varphi_{-1}$ be nondegenerate. Then the bilinear form $Hess_{\x}(\varphi_{-1})$ and the Hermitian Hessian of $\varphi_{-1}$ at $\x$ are nondegenerate. 

In order to interpret the formal intergal (\ref{E:Lerayinv}) in the case when the Hessian of $\varphi_{-1}$ at $\x$ is of the Hermitian type, we will use a formal Gaussian model integral of the Hermitian type on $T_\x M$. Let $U$ be a coordinate chart around $\x$ with coordinates $\{z^p,\bar z^q\}$ and $\{\zeta^p,\bar \zeta^q\}$ be the corresponding coordinates on $T_\x M$. Then
\[
      K(\zeta,\bar\zeta) := -\frac{\p^2 \varphi_{-1}}{\p z^p \p \bar z^q}(\x)\zeta^p \bar \zeta^q
\]
is a potential of the constant pseudo-K\"ahler form 
\[
     \Omega := -i (\p \bar \p  \varphi_{-1})(\x)
\]
on $T_{\x}M$. Let $\gamma = \varphi_0(\x)$ and  $\tau=\rho_0(\x)$ be as above and set $\lambda := m! (2\pi)^m(\tau/\Omega^m)$.  Then the model integral on $T_{\x}M$ will be as follows,
 \[
    \alpha= \nu^{-m} \int_{\{0\}} e^{- \nu^{-1}K + \gamma} \tau = \nu^{-m} \int_{\{0\}} e^{- \nu^{-1}K + \gamma} \lambda \frac{1}{m!} \left(\frac{\Omega}{2\pi}\right)^m = e^\gamma \lambda,
 \]
where we have used (\ref{E:Hermnorm}). Finally, we interpret (\ref{E:Lerayinv}) as
 \[
      \nu^{-m}\int_{\{\x\}} e^{\varphi} f \, \rho = \frac{e^{\varphi_0(\x)}\rho_0(\x)}{\frac{1}{m!}\left(\frac{\Omega}{2\pi}\right)^m}\Lambda(f),
 \]
 where $\Lambda$ is the unique FOI strongly associated with the pair $(\varphi,\rho)$ and such that its leading term is $\Lambda_0=\delta_\x$.

\section{Deformation quantization}

In this section we provide general facts on deformation quantization. If $M$ is a Poisson manifold with Poisson bracket $\{\cdot,\cdot\}$, deformation quantization on $M$ is given by a star product, which is an associative $\C((\nu))$-linear product $\star$ on the space $C^\infty(M)((\nu))$ expressed by a formal bidifferential operator
\begin{equation}\label{E:star}
    f \star g = fg + \sum_{r=1}^\infty \nu^r C_r(f,g)
\end{equation}
such that
\[
      C_1(f,g) - C_1(g,f) = i\{f,g\}.
\]
It is assumed that the unit constant is the unity of the star product,
\[
     f \star 1 = 1 \star f = f
\]
for all $f$. Since star products are given by formal bidifferential operators, they can be restricted to any open subset $U \subset M$. We denote by $L^\star_f$ and $R^\star_f$ the operators of the left and right star multiplication by $f$, respectively, so that
\[
    L^\star_f g = f \star g = R^\star_g f
\]
for any functions $f,g$. We usually drop the superscript $\star$ unless it leads to confusion. The associativity of the star product implies that $[L_f,R_g]=0$ for any $f,g$. 

Two star products $\star$ and $\star'$ on $(M, \{\cdot,\cdot\})$ are equivalent if there exists a formal differential operator $T = 1 + \nu T_1 + \ldots$ on $M$ such that
\[
     f \star' g = T^{-1}(Tf \star Tg).
\]
Deformation quantization was introduced in \cite{BFFLS}. Kontsevich proved in \cite{K} that star products exist on an arbitrary Poisson manifold $M$ and their equivalence classes are parametrized by the formal deformations of the Poisson structure modulo the action of the group of formal paths in the diffeomorphism group of $M$, starting at
the identity diffeomorphism. In the case when $M$ is an arbitrary symplectic manifold, Fedosov gave in \cite{F1} a geometric construction of star products in each equivalence class. 

Let $\star$ be a star product on a symplectic manifold $M$ with symplectic form $\omega_{-1}$ and of dimension $n = 2m$. There exists a formal density $\mu$ globally defined on $M$ such that
\[
    \int f \star g\,  d\mu = \int g \star f\, d\mu
\]
for any functions $f,g$ on $M$ such that $f$ or $g$ has a compact support. It is called a trace density of the star product. All trace densities of the product $\star$ on a connected manifold $M$ form a one dimensional vector space over the field $\C((\nu))$. Since a symplectic manifold is canonically oriented by the Liouville volume form $\omega_{-1}^m/m!$, the density $\mu$ is given by a formal volume form.

 A trace density can be canonically normalized. Let $U \subset M$ be a contractible neighborhood. There exists a $\C$-linear derivation of the star product $\star$ on $U$ of the form
\[
     \delta_\star = \frac{d}{d\nu} + A,
\]
where $A$ is a formal ($\C((\nu))$-linear) differential operator on $U$. It is unique up to an inner $C((\nu))$-linear derivation $\ad_\star (f) = [f, \cdot]_\star$, where $[\cdot,\cdot]_\star$ is the star commutator and $f \in C^\infty(U)((\nu))$. We call it a local $\nu$-derivation of the star product (see \cite{GR99}, where it is called a local $\nu$-Euler derivation).

A trace density $\mu$ is canonically normalized if the following two conditions are satisfied, as shown in \cite{LMP2}.
\begin{enumerate}
\item For every contractible neighborhood $U \subset M$ and a $\nu$-derivation $\delta_\star$ on $U$, the identity
\[
     \frac{d}{d\nu} \int f\, d\mu = \int \delta_\star (f)\, d\mu
\]
holds for every formal function $f$ with compact support on $U$.
\item The leading term of the formal density $\mu$ is
\[
     \frac{1}{m!}\left(\frac{\omega_{-1}}{2\pi \nu}\right)^{m}.
\]

\end{enumerate}
Condition (1) normalizes $\mu$ up to a multiplicative constant from $\C$, which is fixed by condition (2).

It is interesting to notice that this normalization closely resembles the normalization of a formal oscillatory integral considered in this paper.

\section{Deformation quantization with separation of variables}

On K\"ahler manifolds Berezin defined in \cite{Ber1} and \cite{Ber2} a quantization procedure which leads to star products with the separation of variables property (see \cite{E} and \cite{KS}). The Berezin's construction involves explicit integral formulas which have asymptotic expansions as a certain small positive parameter $h$ tends to zero. In our paper we will show that in the framework of deformation quantization with separation of variables one can produce a family of formal oscillatory integrals whose formal oscillatory kernels can be identified. Two such FOIs were already described in \cite{KS}.

Let $M$ be a pseudo-K\"ahler manifold of complex dimension $m$ with a pseudo-K\"ahler form $\omega_{-1}$. A star product $\star$ on $M$ has the property of separation of variables of the anti-Wick type if for every open subset $U \subset M$,
\begin{equation}\label{E:aWick}
     a \star f = af \mbox{ and } f \star b = bf
\end{equation}
for every function $f$, every holomorphic function $a$, and every antiholomorphic function $b$ on $U$. A star product $\star$ has the property of separation of variables of the Wick type if
\[
     b \star f = bf \mbox{ and } f \star a = af,
\]
where $f,a,$ and $b$ are as above.

It was proved in \cite{BW} and \cite{CMP1} that star products with separation of variables exist on an arbitrary pseudo-K\"ahler manifold. Let $\star$ be a star product of the anti-Wick type on $(M,\omega_{-1})$. If $a$ is a local holomorphic function and $b$ is a local antiholomorphic function, then the anti-Wick property (\ref{E:aWick}) means that
\[
      L^\star_a = a \mbox{ and } R^\star_b = b
\]
are pointwise multiplication operators.

It was shown in~\cite{CMP1} that the star products of the anti-Wick type on $(M,\omega_{-1})$ can be bijectively parameterized (not only up to equivalence) by the formal closed $(1,1)$-forms
\[
     \omega = \nu^{-1} \omega_{-1} + \omega_0 + \ldots.
\]
The star product of the anti-Wick type $\star$ parametrized by a classifying form $\omega$ can be described as follows.  Let $U \subset M$ be a contractible coordinate chart with holomorphic coordinates $\{z^k, \bar z^l\}$. There exists a formal potential
\[
     \Phi = \nu^{-1} \Phi_{-1} + \Phi_0 + \ldots
\]
of $\omega$ on $U$, so that $\omega = i\p \bar\p \Phi$. Then the following property holds,
\[
     L_{\frac{\p \Phi}{\p z^k}} = \frac{\p}{\p z^k} + \frac{\p \Phi}{\p z^k} \mbox{ and } R_{\frac{\p \Phi}{\p \bar z^l}} = \frac{\p}{\p \bar z^l} + \frac{\p \Phi}{\p \bar z^l}.
\]
For any formal function $f$ on $U$ there exists a unique formal differential operator $A$ on $U$ such that
it commutes with the operators
\[
       R_{\bar z^l} = \bar z^l \mbox{ and } R_{\frac{\p \Phi}{\p \bar z^l}} = \frac{\p}{\p \bar z^l} + \frac{\p \Phi}{\p \bar z^l}
\]
for all $l$ and satisfies the condition that $A1 = f$. This operator is the left star multiplication operator by $f$, $A=L_f$. Now, $f \star g = Ag$. Therefore, the star product~$\star$ can be recovered from the classifying form~$\omega$ on every contractible chart $U$.

There exists a formal differential operator $I = 1 + \nu I_1 + \ldots$ globally defined on $M$ such that
\[
    I(ab) = b \star a
\]
for every locally defined holomorphic function $a$ and antiholomorphic function $b$. In particular, $Ia=a, Ib=b$, and $I1=1$. It is called the formal Berezin transform of the star product~$\star$. A star product of the anti-Wick type can be recovered from its formal Berezin transform. The equivalent star product
\[
     f \star' g = I^{-1}(If \star Ig)
\]
is a star product of the Wick type on $(M,\omega_{-1})$. The opposite product 
\[
\tilde\star = (\star')^{\mathrm{opp}}
\]
is a star product of the anti-Wick type on $(M,-\omega_{-1})$. It is called the dual star product of $\star$. Its formal Berezin transform is $I^{-1}$. Denote by $\tilde\omega$ its classifying form. Let $U \subset M$ be a contractible coordinate chart. Let $  \Phi = \nu^{-1} \Phi_{-1} + \Phi_0 + \nu \Phi_1 + \ldots$ be a potential of $\omega$ on $U$. As shown in \cite{LMP2}, there exists a potential
\begin{equation}\label{E:Psipot}
       \Psi = -\nu^{-1} \Phi_{-1} + \Psi_0 + \nu\Psi_1 + \ldots
\end{equation}
of the dual form $\tilde\omega$ on $U$ such that
\begin{eqnarray}\label{E:phipsi}
      \frac{\p \Psi}{\p z^k} = - I^{-1} \left(\frac{\p \Phi}{\p z^k}\right) \mbox{ and } \frac{\p \Psi}{\p \bar z^l} = - I^{-1} \left(\frac{\p \Phi}{\p \bar z^l}\right).
\end{eqnarray}
Then
\begin{equation}\label{E:trdens}
     \nu^{-m} e^{\Phi+\Psi} dz^1 \wedge \ldots \wedge dz^m \wedge d\bar z^1 \wedge \ldots \wedge d\bar z^m
\end{equation}
is a (not normalized) trace density of the product $\star$ on $U$. It can be normalized up to a multiplicative constant from $\C$ as follows, as shown in \cite{LMP2} (see also \cite{JGP2016}). There exists a potential $\Psi$ of the dual form $\tilde\omega$ that satisfies (\ref{E:Psipot}), (\ref{E:phipsi}), and the equation
\begin{equation}\label{E:psinu}
     \frac{d \Phi}{d \nu} + I \left(\frac{d \Psi}{d \nu}\right) = \frac{m}{\nu}.
\end{equation}
The potential $\Psi$ is determined by (\ref{E:phipsi}) and (\ref{E:psinu}) up to an additive constant from $\C$. With this choice of $\Phi$ and $\Psi$ the density (\ref{E:trdens}) differs from the canonical trace density by a multiplicative constant from $\C$.

{\it Remark.} The classifying form of the star product of the Wick type~$\star'$ is $-\tilde\omega$. If $\mu$ is the canonical trace density of the star product $\star$, then it is also the canonical trace density of $\star'$. The canonical trace density of the dual star product $\tilde\star$ is $(-1)^m \mu$.

\section{Formal oscillatory integrals $K^{(l)}$}

Let $\star$ be a star product of the anti-Wick type with classifying form $\omega = \nu^{-1}\omega_{-1} + \omega_0 + \ldots$ on a pseudo-K\"ahler manifold $M$ of complex dimension $m$. Let $\x$ be a fixed point in $M$. For any $l \in \N$ we introduce a formal distribution
\begin{equation}\label{E:fnK}
     K^{(l)}(f_1, \ldots, f_l) := \left(If_1 \star \ldots \star If_l\right)(\x) = I(f_1 \star' \ldots \star' f_l)(\x)
\end{equation}
on $M^l$ supported at the point $\x^{(l)}:=(\x, \ldots,\x) \in M^l$. Observe that $K^{(l)}(1, \ldots,1)=1$. In \cite{KS} it was proved that $K^{(1)}$ and $K^{(2)}$ are FOIs associated with certain formal oscillatory kernels. We will show that, for every $l$,  $K^{(l)}$ is a FOI at $\x^{(l)}$ and give its formal oscillatory kernel.

As explained in \cite{KS}, in order to express the formal phase of the FOI $K^{(l)}$, it is convenient to use almost analytic extensions of smooth functions on a complex manifold. 

If $U$ is an open subset of a complex manifold $M$ and $Z$ is a relatively closed submanifold of $U$, a function $f \in C^\infty(U)$ is called almost analytic along $Z$ if $\bar\p f$ vanishes to infinite order at the points of $Z$. Given a function $g \in C^\infty(Z)$, a function $f$ on $U$ almost analytic along $Z$ and such that $f|_Z = g$ is called an almost analytic extension of $g$. Given a function $f(x) \in C^\infty(U)$, there exists a function $\tilde f(x,y)$ on the complex manifold $U\times \xbar{U}$ almost analytic along the diagonal and such that $\tilde f(x,x) = f(x)$. It is also called an almost analytic extension of~$f$. The full jet of $\tilde f$ at any point of the diagonal of $U\times \xbar{U}$ is completely determined by the function $f$. If $f(z, \bar z)$ is a real analytic function on $U$, then its analytic extension $f(z, \bar w)$ on a neighborhood of the diagonal of $U \times \xbar{U}$ is an almost analytic extension of $f$.

Denote by $\mu$ the canonically normalized formal trace density of the star product $\star$. Let $U \subset M$ be a contractible neighborhood of the point $\x$ and $\Phi = \nu^{-1}\Phi_{-1} + \Phi_0 + \ldots$ be a potential of $\omega$ on $U$. Consider an almost analytic extension $\tilde\Phi =  \nu^{-1}\tilde\Phi_{-1} + \tilde\Phi_0 + \ldots$ of the potential $\Phi$ on $U \times \xbar{U}$. For any $l \in \N$ we define a function $F^{(l)} = \nu^{-1} F^{(l)}_{-1} + F^{(l)}_0 + \ldots$ on~$U^l$,
\begin{eqnarray*}
    F^{(l)} (x_1, \ldots,x_l):=\tilde\Phi(x_0, x_1) + \tilde \Phi(x_1, x_2) + \ldots  +\tilde\Phi(x_l,x_0)\\
     - (\Phi(x_0) + \Phi(x_1) + \ldots + \Phi(x_l)).
\end{eqnarray*}
Its full jet at the point $\x^{(l)}\in U^l$ does not depend on the choice of the potential $\Phi$ and on the choice of its
almost analytic extension. We will prove that $K^{(l)}$ is the FOI at $\x^{(l)}$ that admits the following formal oscillatory integral representation,
\begin{equation}\label{E:klint}
     K^{(l)}(f_1, \ldots, f_l) = \int_{\left(\x^{(l)}\right)} e^{F^{(l)}} f_1 \otimes \ldots \otimes f_l\,  \mu^{\otimes l}.
\end{equation}

\begin{proposition}\label{P:nondegcp}
The point $\x^{(l)}\in M^l$ is a nondegenerate critical point of the Hermitian type of the function $F^{(l)}_{-1}$ with zero critical value. 
\end{proposition}
\begin{proof}
Let $U \subset M$ be a contractible coordinate neighborhood of the point $\x$ with holomorphic coordinates $\{z^p,\bar z^q\}$, $\Phi$ be a potential of $\omega$ on $U$, and $\tilde\Phi$ be its almost analytic extension on $U \times \xbar{U}$. Since $\tilde\Phi_{-1}(\x,\x) = \Phi_{-1}(\x)$, it follows that $F^{(l)}_{-1}(\x^{(l)})=0$. Set
\begin{equation}\label{E:gpqx}
     g_{pq} = \frac{\p^2 \Phi_{-1}}{\p z^p \p\bar z^q}(\x).
\end{equation}
Denote by $\{z^p_i,\bar z^q_i\}$ the corresponding coordinates on the $i$th copy of $U$ in $U^l$. Then
\[
     \frac{\p}{\p z^p_i} \Phi_{-1}(x_a) \Big |_{x_a=x_0} = \begin{cases} \frac{\p \Phi_{-1}}{\p z^p}(\x) & \mbox{ if } i=a;\\
     0 & \mbox{ otherwise}.
     \end{cases}
\]
Also,
\[
     \frac{\p}{\p z^p_i} \tilde\Phi_{-1}(x_a,x_b) \Big |_{x_a=x_b=x_0} = \begin{cases} \frac{\p \Phi_{-1}}{\p z^p}(\x) & \mbox{ if } i=a;\\
     0 & \mbox{ otherwise}.
     \end{cases}
\]
One can prove similar formulas for the antiholomorphic derivatives of $\Phi_{-1}$ and $\tilde\Phi_{-1}$.
It follows from these formulas that
\[
      \frac{\p F^{(l)}_{-1}}{\p z^p_i}\left(\x^{(l)}\right) =0 \mbox{ and } \frac{\p F^{(l)}_{-1}}{\p \bar z^q_i} \left(\x^{(l)}\right)=0
\]
for all $i,p,q$. Therefore, $\x^{(l)}$ is a critical point of $F^{(l)}_{-1}$. We have
\[
     \frac{\p^2}{\p z^p_i \p z^s_j} \Phi_{-1}(x_a) \Big |_{x_a=x_0} = \begin{cases} \frac{\p^2 \Phi_{-1}}{\p z^p \p z^s}(\x) & \mbox{ if } i=j=a;\\
     0 & \mbox{ otherwise}.
     \end{cases}
\]
Similarly,
\[
     \frac{\p^2}{\p z^p_i \p z^s_j} \tilde\Phi_{-1}(x_a,x_b) \Big |_{x_a=x_b=x_0} = \begin{cases} \frac{\p^2 \Phi_{-1}}{\p z^p \p z^s}(\x) & \mbox{ if } i=j= a;\\
     0 & \mbox{ otherwise}.
     \end{cases}
\]
We have similar formulas for the antiholomorphic second order partial derivatives of $\Phi_{-1}$ and $\tilde\Phi_{-1}$. These formulas imply that
\[
    \frac{\p^2 F^{(l)}_{-1}}{\p z^p_i \p z^s_j} \left(\x^{(l)}\right)=  \frac{\p^2 F^{(l)}_{-1}}{\p \bar z^q_i \p \bar z^t_j}\left(\x^{(l)}\right) = 0
\]
for all $i,j,p,q,s,t$.
Therefore, the Hessian of $F^{(l)}_{-1}$ at $\x^{(l)}$ is of the Hermitian type. We have
\[
     \frac{\p^2}{\p z^p_i \p \bar z^q_j} \Phi_{-1}(x_a) \Big |_{x_a=x_0} = \begin{cases} g_{pq} & \mbox{ if } i=j=a;\\
     0 & \mbox{ otherwise}.
     \end{cases}
\]
Similarly,
\[
     \frac{\p^2}{\p z^p_i \p \bar z^q_j} \tilde\Phi_{-1}(x_a,x_b) \Big |_{x_a=x_b=x_0} = \begin{cases} g_{pq} & \mbox{ if } \mbox{ if } i=a \mbox{ and } j=b;\\
     0 & \mbox{ otherwise}.
     \end{cases}
\]
It follows from these formulas that  the Hermitian Hessian  of $F^{(l)}_{-1}$ at $\x^{(l)}$ has the entries
\begin{equation}\label{E:HermHess}
    \frac{\p^2 F^{(l)}_{-1}}{\p z^p_i \p \bar z^q_j} \left(\x^{(l)}\right) = \begin{cases} - g_{pq} & \mbox{ if } 1 \leq i=j \leq l;\\
    g_{pq} & \mbox{ if } 1 \leq i=j-1 \leq l-1;\\
    0 & \mbox{ otherwise}.
    \end{cases}
\end{equation}
Therefore it is block-triangular with the diagonal blocks $(-g_{pq})$, whence one can see that the Hessian of $F^{(l)}_{-1}$ at $\x^{(l)}$ is nondegenerate.
\end{proof}

In \cite{KS} it was proved that $K^{(1)}$ is a FOI at $\x$ associated with the pair $(F^{(1)}, \mu)$. One can show along the same lines that $K^{(l)}$ is a FOI at $\x^{(l)}$ associated with the pair
\begin{equation}\label{E:klpair}
          \left(F^{(l)}, \mu^{\otimes l}\right)
\end{equation}
for all $l$. Below we will prove that $K^{(l)}$ is strongly associated with the pair~(\ref{E:klpair}).

Given a potential $\Phi$ of the classifying form $\omega$ on a neighborhood $U \subset M$, it was shown in \cite{LMP1} that
\[
    \delta^l_\star= \frac{d}{d\nu} + \frac{d \Phi}{d\nu} - L^\star_{\frac{d \Phi}{d\nu}} \mbox{ and } \delta^r_\star =  \frac{d}{d\nu} + \frac{d \Phi}{d\nu} - R^\star_{\frac{d \Phi}{d\nu}} 
\]
are local $\nu$-derivations of the star product $\star$.
\begin{lemma}\label{L:deldel}
Given a potential $\Phi$ of the form $\omega$ on a contractible coordinate chart $U \subset M$ and a potential $\Psi$ of the dual form $\tilde\omega$ satisfying (\ref{E:Psipot}), (\ref{E:phipsi}), and (\ref{E:psinu}), the following identity holds for all $f \in C^\infty(U)((\nu))$,
\begin{equation}\label{E:deltaprime}
        \delta^r_\star I(f) = I \left(\delta^l_{\star'} f\right),
\end{equation}
where
\[
     \delta^l_{\star'} = \frac{d}{d\nu} + \frac{d\Psi}{d\nu} - L^{\star'}_{\frac{d\Psi}{d\nu}}
\]
is a local $\nu$-derivation of the star product $\star'$.
\end{lemma}
\begin{proof}
Suppose that a formal function $f \in C^\infty(U)((\nu))$ is factorized as $f = ab$, where $a$ is formal holomorphic and $b$ is formal antiholomorphic on $U$. Then
\begin{eqnarray*}
  \delta^r_\star I(f) = \delta^r_\star(b \star a)  = \delta^r_\star (b) \star a + b \star \delta^r_\star (a) =\hskip 2cm\\
  \left(\frac{d b}{d\nu} + \frac{d \Phi}{d\nu}b  - R_{\frac{d \Phi}{d\nu}}b\right) \star a + b \star \left(\frac{d a}{d\nu} + \frac{d \Phi}{d\nu}a  - R_{\frac{d \Phi}{d\nu}}a\right) =\\
   \left(\frac{d b}{d\nu} + \frac{d \Phi}{d\nu} \star b  - b \star \frac{d \Phi}{d\nu}\right) \star a + b \star \left(\frac{d a}{d\nu} + a \star \frac{d \Phi}{d\nu}  - a \star \frac{d \Phi}{d\nu}\right) =\\
   I \left(\frac{d b}{d\nu} a + I^{-1}\left(\frac{d \Phi}{d\nu}\right) \star' b \star' a- b \star ' I^{-1}\left(\frac{d \Phi}{d\nu}\right) \star' a + b \frac{da}{d\nu}\right)=\\
   I\left(\frac{df}{d\nu} + I^{-1}\left(\frac{d \Phi}{d\nu}\right) \star' f - I^{-1}\left(\frac{d \Phi}{d\nu}\right) f\right) = \\
    I\left(\frac{df}{d\nu} + \left(\frac{m}{\nu} - \frac{d \Psi}{d\nu}\right) \star' f - \left(\frac{m}{\nu}-\frac{d \Psi}{d\nu}\right) f\right) = I \left(\delta^l_{\star'} f\right).
\end{eqnarray*}
The identity (\ref{E:deltaprime}) holds for any $f \in C^\infty(U)((\nu))$ because $I$ is a formal differential operator.
\end{proof}
Given a function $F$ on $M^l$, sometimes we will express $K^{(l)}(F)$ as
\[
    K^{(l)}(F(x_1, \ldots,x_l)).
\]
\begin{lemma}\label{L:gK}
Let $\tilde g(x,y)$ be an almost analytic extension of a function $g(x)$ on $M$. Then
\begin{eqnarray*}
K^{(l)}(\tilde g(x_i,x_{i+1})f_1(x_1) \ldots f_l(x_l)) =\hskip 2cm\\
K^{(l)}(f_1, \ldots, f_i,  I^{-1}(g) \star' f_{i+1}, f_{i+2}, \ldots, f_l)
\end{eqnarray*}
for $0 \leq i \leq l-1$ and
\begin{eqnarray*}
K^{(l)}(\tilde g(x_l,x_0)f_1(x_1) \ldots f_l(x_l)) =
K^{(l)}(f_1, \ldots,  f_{l-1}, f_l \star' I^{-1}(g)).
\end{eqnarray*}
\end{lemma}
\begin{proof}
Suppose that $U$ is a coordinate chart containing $\x$. Let $g = ab$ on $U$, where $a$ is holomorphic and $b$ antiholomorphic. Then $\tilde g(x,y) := a(x)b(y)$ is the holomorphic extension of $g$. For $i=0$,
\begin{eqnarray*}
K^{(l)}(\tilde g(x_0,x_1)f_1(x_1) \ldots f_l(x_l)) =K^{(l)}(a(x_0)b(x_1)f_1(x_1) \ldots f_l(x_l)) =\\
a(x_0)( I(b \star' f_1 \star' \ldots \star' f_l))(\x)= (a \star I(b \star' f_1 \star' \ldots \star' f_l))(\x)=\\
I(a \star' b \star' f_1 \star' \ldots \star' f_l))(\x)= I\left((I^{-1}(ab)\star' f_1) \star'  f_2 \star' \ldots \star' f_l\right)(\x)=\\
I\left((I^{-1}(g) \star' f_1) \star'  f_2 \star' \ldots \star' f_l\right)(\x) = K^{(l)}\left(I^{-1}(g) \star' f_1, f_2, \ldots, f_l\right).
\end{eqnarray*}
For $1 \leq i \leq l-1$,
\begin{eqnarray*}
  K^{(l)}(\tilde g(x_i,x_{i+1})f_1(x_1) \ldots f_l(x_l)) =\hskip 3cm \\
  K^{(l)} (f_1(x_1) \ldots (f_i(x_i)a(x_i) )(b(x_{i+1}) f_{i+1}(x_{i+1})) \ldots f_l(x_l)) =\\
  I(f_1 \star' \ldots \star' f_i \star' (a \star' b \star' f_{i+1} )\star' \ldots \star' f_l )(\x)=\\
    I\left(f_1 \star' \ldots \star' f_i \star' (I^{-1}(ab) \star' f_{i+1} )\star' \ldots \star' f_l \right)(\x)=\\
     I\left(f_1 \star' \ldots \star' f_i \star' (I^{-1}(g) \star' f_{i+1} )\star' \ldots \star' f_l \right)(\x)=\\
     K^{(l)}(f_1, \ldots, f_i,  I^{-1}(g) \star' f_{i+1}, \ldots, f_l).
\end{eqnarray*}
We have shown that the first equality of the lemma holds when $g=ab$. For any integer $N \geq 0$, the first equality modulo $\nu^N$ depends on the jet of $g$ of finite order $K$  at $\x$. We can approximate $g$ by a finite sum $\sum_{k} a_kb_k$, where $a_k$ are holomorphic and $b_k$ are antiholomorphic so that $g$ and $\sum_{k} a_kb_k$ will have the same jet of order $K$ at $\x$. It follows that the first equality holds modulo $\nu^N$ for any $N$ and therefore is true. The proof of the second equality of the lemma is similar to the proof of the first one when $i=0$.
\end{proof}
\begin{proposition}\label{P:klstrong}
  The FOI $K^{(l)}$ at $\x^{(l)}$ is strongly associated with the pair~(\ref{E:klpair}).
\end{proposition}
\begin{proof}
Let $U \subset M$ be a contractible coordinate chart and $\Phi$ be a potential of the classifying form $\omega$ on $U$. Set
\[
u^{(l)} := \Phi (x_1) + \Psi(x_1) + \ldots + \Phi(x_l) + \Psi(x_l), 
\]
where $\Phi$ and $\Psi$ satisfy  (\ref{E:Psipot}), (\ref{E:phipsi}), and (\ref{E:psinu}). 
Then
\[
    \mu^{\otimes l} = \nu^{-ml} e^{u^{(l)}} dx^{(l)}
\]
for a  properly normalized Lebesgue volume form $dx^{(l)}$ on~$U^l$. We need to check that, according to formula (\ref{E:stronga}) with $a=1$,
\begin{equation}\label{E:twel}
\frac{d}{d \nu} K^{(l)} (F) = K^{(l)} \left( \left(\frac{d}{d\nu} + \frac{d (F^{(l)} + u^{(l)})}{d\nu} - \frac{ml}{\nu}\right)F\right).
\end{equation}
Using Lemma \ref{L:gK}, we get that the right-hand side of (\ref{E:twel}) equals
\begin{eqnarray*}
K^{(l)} \left(\left(\frac{d}{d\nu}  - \frac{d\Phi(x_0)}{d\nu} + \sum_{i=1}^{l}\frac{d\Psi(x_i)}{d\nu} - \frac{ml}{\nu}\right)F\right) +\\
K^{(l)} \left(\left(\sum_{i=0}^{l-1}\frac{d\tilde\Phi(x_i,x_{i+1})}{d\nu}   + \frac{d\tilde\Phi(x_l,x_0)}{d\nu}\right)F\right) =\\
K^{(l)} \left(\left(\frac{d}{d\nu}  - \frac{d\Phi(x_0)}{d\nu} + \sum_{i=1}^{l}\frac{d\Psi(x_i)}{d\nu} - \frac{ml}{\nu}\right)F\right) +\\
\sum_{i=0}^{l-1}K^{(l)} \left(f_1, \ldots, I^{-1}\left(\frac{d\Phi}{d\nu}\right) \star' f_{i+1}, \ldots,f_l\right)+\\
K^{(l)} \left(f_1,  \ldots,f_l \star' I^{-1}\left(\frac{d\Phi}{d\nu}\right) \right).
\end{eqnarray*}
It follows from Lemma \ref{L:deldel} that
\begin{equation}\label{E:diid}
\delta_\star^r I(f_1 \star' \ldots \star' f_l) = \sum_{i=1}^l I \left(f_1 \star' \ldots \star' \delta^l_{\star'} f_i \star' \ldots \star' f_l\right). 
\end{equation}
Using (\ref{E:psinu}), we obtain that
\begin{eqnarray}\label{E:RdPhi}
R^\star_{\frac{d\Phi}{d \nu}}I(f_1 \star' \ldots \star' f_l) = I\left(f_1 \star' \ldots \star' f_l \star' I^{-1}\left(\frac{d\Phi}{d \nu}\right) \right) =\nonumber\\
\frac{m}{\nu} I(f_1 \star' \ldots \star' f_l) - I\left(f_1 \star' \ldots \star' f_l \star' \frac{d\Psi}{d \nu} \right).
 \end{eqnarray}
 We use the explicit formulas for $\delta_\star^r$ and $\delta_{\star'}^l$ in (\ref{E:diid}), take into account (\ref{E:RdPhi}), evaluate (\ref{E:diid}) at $\x$, and set $F = f_1 \otimes \ldots \otimes f_l$, arriving at the equality
\begin{eqnarray}\label{E:expdd}
\left(\frac{d}{d \nu} + \frac{d\Phi}{d \nu} - \frac{m}{\nu}\right)K^{(l)}(F) + K^{(l)}\left(f_1, \ldots, f_l \star' \frac{d\Psi}{d \nu}\right)=\nonumber\\
 K^{(l)} \left(\left(\frac{d}{d \nu} + \sum_{i=1}^l \frac{d\Psi(x_i)}{d \nu}\right) F\right) -\hskip 2cm\\
\sum_{i=1}^l K^{(l)} \left(f_1, \ldots, \frac{d\Psi}{d\nu} \star'  f_i, \ldots,f_l\right).\nonumber
\end{eqnarray}
Now formula (\ref{E:twel}) follows from (\ref{E:expdd}) and (\ref{E:psinu}).
\end{proof}

Finally, we can completely justify formula (\ref{E:klint}).

\begin{theorem}
Let $\mu$ be the canonically normalized trace density of the star product of the anti-Wick type $\star$ with classifying form $\omega$ on a pseudo-K\"ahler manifold $M$ of complex dimension $m$ and $\x$ be a point in $M$. Then the full jet of the formal integral kernel of the FOI $K^{(l)}$ at the point $\x^{(l)} \in M^l$ is given by the formula
\begin{equation}\label{E:klfok}
      e^{F^{(l)}} \mu^{\otimes l}.
\end{equation}
\end{theorem}
\begin{proof}
Proposition \ref{P:klstrong} implies that (\ref{E:klfok}) is the formal oscillatory kernel of the FOI $K^{(l)}$ at $\x^{(l)} \in M^l$ up to a nonzero multiplicative constant $\alpha \in \C$. Thus the leading term of the formal integral in (\ref{E:klint}) is 
\[
\alpha\delta_{\x^{(l)}}(f_1 \otimes \ldots \otimes f_l) = \alpha f_1(\x) \cdot \ldots \cdot f_l(\x).
\]
Since $K^{(l)}(1,\ldots,1)=1$, it remains to show that $\alpha=1$. The constant $\alpha$ is given by the model formal Gaussian integral of the Hermitian type on $T_{\x^{(l)}} M^l$. We will use notations introduced in Proposition \ref{P:nondegcp}. Let $U \subset M$ be a contractible chart  containing the point $\x$ with holomorphic coordinates $\{z^p,\bar z^q\}$.
We denoted by $\{z^p_i,\bar z^q_i\}$ the corresponding coordinates on the $i$th copy of $U$ in $U^l$. Now we need the corresponding coordinates $\{w^p_i,\bar w^q_i\}$ on the tangent space $T_{\x^{(l)}} M^l = (T_\x M)^l$. We define (1,1)-forms on the tangent space $T_{\x^{(l)}} M^l$,
\[
     \omega_{-1}^{(j,s)} := i g_{pq} w_j^p \bar w_s^q,
\]
where $g_{pq}$ is given by (\ref{E:gpqx}) and $1 \leq j,s \leq l$. The Hermitian Hessian of the function $F^{(l)}_{-1}$ at $\x^{(l)}$ given in (\ref{E:HermHess}) induces the form (taken with the opposite sign)
\[
   \Omega = \sum_{j=1}^l \omega_{-1}^{(j,j)} - \sum_{j=1}^{l-1} \omega_{-1}^{(j,j+1)}
\]
on $T_{\x^{(l)}} M^l$. The leading term of the volume form $\mu^{\otimes l}$ induces the volume form 
\[
    \prod_{j=1}^l  \frac{1}{m!}\left(\frac{\omega_{-1}^{(j,j)}}{2\pi\nu}\right)^m
\]
on $T_{\x^{(l)}}M^l$. In order to prove that $\alpha=1$, we observe first that since $F^{(l)}(\x^{(l)})=0$, the term $F_0^{(l)}(\x^{(l)})$ does not contribute to the model integral. We see from the normalization condition (\ref{E:Hermnorm}) that it suffices to verify the identity
\begin{equation}\label{E:verif}
    \frac{1}{(ml)!} \left(\frac{\Omega}{2\pi\nu}\right)^{ml} = \prod_{j=1}^l \frac{1}{m!}\left(\frac{\omega_{-1}^{(j,j)}}{2\pi\nu}\right)^m.
\end{equation}
The summands that can nontrivially contribute to the volume form $\Omega^{ml}$ are signed products of $ml$ forms $\omega_{-1}^{(j,s)}$ such that each index $j$ and each index $s$ appear exactly $m$ times. Explicitly, each summand is of the form
\[
   \prod_{j=1}^l  \prod_{k=1}^m(-1)^{s_{jk}-j}\omega_{-1}^{(j, s_{jk})},
\]
where $s_{jk} = j$ or $s_{jk}=j+1$. Since the multiset $\{s_{11}, \ldots, s_{lm}\}$ should contain each number $\{1,2, \ldots,l\}$ with multiplicity $m$, we should have $s_{jk}=j$ for all $j$. By the multinomial theorem,
\[
     \Omega^{ml} = \left( \sum_{j=1}^l \omega_{-1}^{(j,j)}\right)^{ml} = \frac{(ml)!}{(m!)^l} \prod_{j=1}^l \left(\omega_{-1}^{(j,j)}\right)^m,
\]
whence the identity (\ref{E:verif}) follows.
\end{proof}

\section{Appendix}

Let $x^1, \ldots, x^n$ be formal variables. We will use the notations
\[
    \C[[x]] := \C[[x^1, \ldots,x^n]] \mbox{ and } dx := dx^1 \wedge \ldots \wedge dx^n.
\]
In the Appendix we will prove that for an arbitrary pair $(\varphi,dx)$ such that $\x=0$ is a nondegenerate critical point of $\varphi_{-1}$ with zero critical value, there exists a FOI at $\x=0$ associated with this pair. We work with the full jets of functions which are expressed as formal series in~$x^i$.

Let $h_{ij}$ be an invertible symmetric $n \times n$ matrix with constant complex entries and $h^{ij}$ be its inverse matrix.
Set
\[
      \psi = \frac{1}{2}h_{ij}x^ix^j \mbox{ and } \Delta = -\frac{1}{2}h^{ij} \frac{\p^2}{\p x^i \p x^j}.
\]
\begin{lemma}\label{P:delta}
The functional
\[
      \Lambda(f) = \exp \left(\nu \Delta\right) f  |_{x=0}
\]
is the unique FOI at $\x=0$ associated with the pair $(\nu^{-1}\psi, dx)$ and such that $\Lambda(1)=1$.
\end{lemma}
\begin{proof}
We see that $\Lambda(1)=1$ and $\Lambda = \Lambda_0 + \nu \Lambda_1 + \ldots$ is a formal distribution supported at $\x=0$ such that $\Lambda_0(f)=f(0)$. We have to prove that condition (\ref{E:axiom}) is satisfied, i.e., for any $i = 1, \ldots,n$ and $f \in \C[[x]]$,
\[
      \Lambda\left(\frac{\p f}{\p x^i} + \frac{1}{\nu}\frac{\p \psi}{\p x^i}f\right)=0.
\]
Since $\p \psi/ \p x^i = h_{ij}x^j$, it reduces to the equations $\Lambda_0 \left(h_{ij} x^j f\right) = 0$ and 
\begin{equation}\label{E:rec1}
    \Lambda_r \left(h_{ij} x^j f\right) = - \Lambda_{r-1} \left(\frac{\p f}{\p x^i}\right)
\end{equation}
for $r \geq 1$, where
\[
      \Lambda_r (f) = \frac{1}{r!} \Delta^r f \Big |_{x=0}.
\]
We have
\[
      \left[\Delta, h_{ij} x^j \right] = - \frac{\p}{\p x^i}.
\]
Therefore,
\begin{equation}\label{E:rec2}
   \left[ \frac{1}{r!} \Delta^r, h_{ij} x^j \right] = - \frac{1}{(r-1)!} \Delta^{r-1} \frac{\p}{\p x^i}.
\end{equation}
Applying both sides of (\ref{E:rec2}) to $f$ and setting $x=0$, we obtain (\ref{E:rec1}).
\end{proof}

Below we use the standard grading on the variables $x^i$ and $\nu$,
\[
    |x^i|=1 \mbox{ and } |\nu|=2.
\]
Consider the subspace $\K \subset\C[[\nu^{-1},\nu,x]]$ of elements 
  \[
 f (\nu,x)= \sum_{r=-\infty}^\infty \nu^r f_r(x)
 \]
 satisfying the following two conditions:
 \begin{enumerate}
 \item For each element $f \in \K$ there is an integer $N \in \Z$ such that $\deg f_r \geq N - 2r$ for all $r \in \Z$.  
 Then we say that the filtration degree of $f$ is at least $N$.

 \item Each homogeneous component of $f$ lies in $\C[\nu^{-1},\nu,x]$.   
 \end{enumerate}
 The following statements are easy to verify.
 \begin{itemize}
 \item The space $\K$ is a filtered algebra with respect to the ``pointwise" product with the descending filtration induced by the standard grading. 
 \item The evaluation mapping
  \[
      f(\nu,x) \mapsto f(\nu,0)
  \]
maps $\K$ to $\C((\nu))$.
\item If $h^{ij}$ is a matrix with constant entries, the operator
  \[
     \exp (\nu \Delta)= \exp \left(-\frac{\nu}{2} h^{ij}\frac{\p^2}{\p x^i \p x^j}\right)
  \]
  leaves $\K$ invariant. 
  \item The space $\K$ contains $\C[[x]]((\nu))$.
  \item If an element $f \in \C[[x]]$ has the filtration degree at least three, then
  \[
      \exp \left(\nu^{-1} f\right) \in \K.
  \]
  \end{itemize}
  Let $\nu^{-1}\varphi_{-1}\in \C[[x]]((\nu))$ be a formal phase function such that $\x=0$ is a nondegenerate critical point of $\varphi_{-1}$. Set
  \[
      h_{ij}= \frac{\p^2 \varphi_{-1}}{\p x^i \p x^j}(0).
  \]
  The filtration degree (the order of zero) of
  \[
      \chi := \varphi_{-1}(x) - \varphi_{-1}(0) - \frac{1}{2}h_{ij}x^i x^j 
  \]
  is at least three. We represent $\nu^{-1}\varphi_{-1}$ as
   \begin{equation}\label{E:repphi}
    \frac{1}{\nu}\varphi_{-1}= \frac{1}{\nu} \varphi_{-1}(0) + \frac{1}{2\nu}h_{ij}x^i x^j + \frac{1}{\nu}\chi
 \end{equation}
  The filtration degree of $\nu^{-1}\chi$ is at least one.
  \begin{lemma}\label{P:tilde}
  The formula
  \[
      \tilde\Lambda(f) := \exp \left(\nu\Delta\right) \left(e^{\nu^{-1}\chi} f\right)\Big |_{x=0}
  \]
  gives a FOI at $\x=0$ associated with the pair
 \begin{equation}\label{E:psipair}
     (\nu^{-1}\varphi_{-1}, dx).
 \end{equation}
Its leading term is $\tilde\Lambda_0(f)=f(0)$.
\end{lemma}
\begin{proof}
We want to prove that
$\tilde\Lambda = \tilde\Lambda_0 + \nu \tilde\Lambda_1 + \ldots$ is a formal distribution supported at zero and $\tilde\Lambda_0(f)=f(0)$.
   Assume that $f \in \C[[x]]$. Consider the term
   \begin{equation}\label{E:klterm}
         \frac{1}{k!}(\nu\Delta)^k \left(\frac{1}{l!}(\nu^{-1} \chi)^l f \right)\Big|_{x=0} = \frac{\nu^{k-l}}{k! l!} \Delta^k (\chi^l f) \Big |_{x=0}
   \end{equation}
 of $\tilde\Lambda(f)$. Since the order of zero of $(\nu^{-1} \chi)^l f$ is at least $3l$, we see that the term (\ref{E:klterm}) is nonzero only if $2k \geq 3l$. Therefore, $\tilde\Lambda = \tilde\Lambda_0 + \nu \tilde\Lambda_1 + \ldots$, where
 \[
    \tilde\Lambda_r(f) = \sum_{k=0}^{3r} \sum_{l=0}^{\lfloor\frac{2k}{3}\rfloor} \frac{1}{k! l!}\Delta^k (\chi^l f) \Big|_{x=0}
\]
for $r\geq 0$. In particular, $\tilde\Lambda_0(f)=f(0)$.

Now we have to show that
\[
     \tilde\Lambda\left(\frac{\p f}{\p x^i} + \frac{1}{\nu}\frac{\p \varphi_{-1}}{\p x^i}f\right)=0
\]
for any $i=1, \ldots,n$ and $f \in \C[[x]]((\nu))$. By representation (\ref{E:repphi}) and Lemma \ref{P:delta},
\begin{eqnarray*}
  \tilde\Lambda\left(\frac{\p f}{\p x^i} +  \frac{1}{\nu}\frac{\p \varphi_{-1}}{\p x^i}f\right)= \tilde\Lambda\left(\frac{\p f}{\p x^i} + \frac{1}{\nu}h_{ij}x^j f + \frac{1}{\nu}\frac{\p \chi}{\p x^i}f\right)=\\
  \exp (\nu\Delta) \left(e^{\nu^{-1}\chi}\left(\frac{\p f}{\p x^i} + \frac{1}{\nu}h_{ij}x^j f + \frac{1}{\nu}\frac{\p \chi}{\p x^i}f\right)\right)=\\
  \exp (\nu\Delta) \left(\left(\frac{\p}{\p x^i} +  \frac{1}{\nu}h_{ij}x^j \right)e^{\nu^{-1}\chi} f\right)= 0.
\end{eqnarray*}
\end{proof}
Using Lemma \ref{P:tilde} and the same arguments as in the second half of its proof, we can prove the following theorem.
\begin{theorem}
Let $\varphi = \nu^{-1}\varphi_{-1} + \varphi_0 + \ldots \in \C[[x]]((\nu))$ be a formal phase function such that $\x=0$ is a nondegenerate critical point of $\varphi_{-1}$. Set $\tilde\varphi := \varphi- \nu^{-1}\varphi_{-1} - \varphi_0(0)$.
Let $\tilde\Lambda$ be the functional built from the phase function $\nu^{-1}\varphi_{-1}$ as in Lemma \ref{P:tilde}. The formula
  \[
      \hat\Lambda(f) :=  \tilde\Lambda \left(e^{\tilde\varphi} f\right)\Big |_{x=0}
  \]
  gives a FOI at $\x=0$ associated with the pair $(\varphi, dx)$. Its leading term is $\hat\Lambda_0(f)=f(0)$.
  \end{theorem}


\begin{thebibliography}{99}
\bibitem{BFFLS} Bayen, F., Flato, M., Fronsdal, C.,
Lichnerowicz, A., and Sternheimer, D.: Deformation theory and
quantization. I. Deformations of symplectic structures. {\it
Ann. Physics} {\bf 111} (1978), no. 1, 61 -- 110.
\bibitem{Ber1} Berezin, F.A.: Quantization. Math. USSR-Izv. {\bf 8} (1974), 1109--1165.
\bibitem{Ber2} Berezin, F.A.: Quantization in complex symmetric spaces. Math. USSR-Izv. {\bf 9} (1975), 341--379.
\bibitem{BW} Bordemann, M., Waldmann, S.: A Fedosov star product
of the Wick type for K\"ahler manifolds. {\it Lett. Math. Phys.}
{\bf 41} (3) (1997), 243 -- 253.
\bibitem{E} Engli\v s, M.: A Forelli-Rudin construction and asymptotics of weighted Bergman kernels. {\it J. Funct. Anal.} {\bf 177} (2000), 257--281.
\bibitem{F1} Fedosov, B.:  A simple geometrical construction of deformation quantization.
{\it J. Differential Geom.} {\bf 40}  (1994),  no. 2, 213--238.
\bibitem{GR99} Gutt, S. and Rawnsley, J.: Equivalence of star products on a symplectic manifold. {\it J. Geom. Phys.} {\bf 29} (1999), 347 -- 392.
\bibitem{CMP1} Karabegov, A.: Deformation quantizations with separation of variables on a K\"ahler manifold.
{\it Commun. Math. Phys.} {\bf 180}  (1996),  no. 3, 745--755.
\bibitem{LMP1} Karabegov A.V.: Cohomological  classification  of  deformation  quantizations  with  separation  of
variables. {\it Lett. Math. Phys.} {\bf 43} (1998), 347--357.
\bibitem{LMP2} Karabegov A.V.: On the canonical normalization of a trace density of deformation quantization. {\it Lett. Math. Phys.}{\bf 45}
(1998), 217 -- 228.
\bibitem{JGP2016} Karabegov A.: On the phase form of a deformation quantization with separation of variables. {\it J. Geom. Phys.} {\bf 104} (2016), 30 -- 38.
\bibitem{KS} Karabegov, A., Schlichenmaier, M.: Identification of Berezin-Toeplitz deformation quantization. {\it J. reine angew. Math.} {\bf 540} (2001), 49 -- 76.
\bibitem{K} Kontsevich, M.: Deformation quantization of Poisson manifolds, I.  {\it Lett. Math. Phys.} {\bf 66} (2003), 157 -- 216.
\bibitem{L} Leray, J.: Lagrangian analysis and quantum mechanics: a mathematical structure related to asymptotic expansions and the Maslov index, {\it MIT press, Cambridge, MA} (1981).
\bibitem{V} Voronov, Th. Th.: Microformal geometry and homotopy algebras,  to appear in {\it Tr. Mat. Inst. Steklova}, \begin{verbatim}http://www.mathnet.ru/php/archive.phtml?wshow=
paper&jrnid=tm&paperid=3934&option_lang=eng\end{verbatim}	
arXiv:1411.6720.
\end{thebibliography}
\end{document}